\documentclass{article}

\usepackage{multicol}
\usepackage{amsmath,amsfonts,amsthm,amssymb}
\usepackage{graphicx}
\usepackage{tikz}
\usepackage{enumitem}
\usepackage{verbatim}
\usepackage{color}
\usepackage{xcolor}
\usepackage[normalem]{ulem}
\usepackage{comment,nicefrac}
\usepackage{caption}
\usepackage{cancel}
\usepackage[english]{babel}
\usepackage{soul}
\usepackage[color=mnt]{todonotes}
\usepackage{subcaption}
\usepackage{setspace}

\pgfdeclarelayer{background}
\pgfsetlayers{background,main}

\newtheorem{thm}{Theorem}

\newtheorem{cnj}[thm]{Conjecture}
\newtheorem{cor}[thm]{Corollary}
\newtheorem{fct}[thm]{Fact}
\newtheorem{lem}[thm]{Lemma}

\newtheorem{prb}[thm]{Problem}
\newtheorem{prp}[thm]{Proposition}

\def\a{{\alpha}}

\def\p{{\pi}}

\def\s{{\sigma}}
\def\t{{\tau}}

\def\cA{{\cal A}}
\def\cB{{\cal B}}
\def\cC{{\cal C}}

\def\cP{{\cal P}}

\def\sP{{\sf P}}

\def\zN{{\mathbb N}}

\def\So{{\mathfrak S}} 

\def\dC{{\dot{C}}}
\def\dD{{\dot{D}}}
\def\dF{{\dot{F}}}

\def\hC{{\hat{C}}}

\def\cost{{\sf cost}}
\def\deg{{\sf deg}}
\def\diam{{\sf diam}}
\def\dist{{\sf dist}}

\def\supp{{\sf supp}}

\def\sP2P{{$\Pi_2^{\sP}$}}

\def\sse{{\ \subseteq\ }}

\def\sqr#1#2{{\vcenter{\hrule height.#2pt
        \hbox{\vrule width.#2pt height#1pt \kern#1pt
                \vrule width.#2pt}
        \hrule height.#2pt}}}
\def\gbox{{\mathchoice\sqr45\sqr45\sqr{2.1}3\sqr{1.5}3}}

\definecolor{brwn}{RGB}{140, 70, 20}
\definecolor{gren}{RGB}{  0,140, 10}
\definecolor{purp}{RGB}{147, 51,255}
\definecolor{mnt}{RGB}{167,230,215}

\newcommand{\ma}[1]{\textcolor{blue}{\sf{#1}}}
\newcommand{\vb}[1]{\textcolor{gren}{\sf{#1}}}
\newcommand{\gh}[1]{\textcolor{brwn}{\sf{#1}}}
\newcommand{\up}[1]{\textcolor{red}{\sf{#1}}}
\newcommand{\yb}[1]{\textcolor{purp}{\sf{#1}}}

\newcommand{\rev}[1]{\textcolor{black}{#1}}

\newcommand{\ed}[1]{\textcolor{black}{#1}}

\begin{document}

\title{Target Pebbling in Trees}

\author{
Matheus Adauto
    \thanks{
        Department of Mathematics and Applied Mathematics, Virginia Commonwealth University, Richmond, Virginia, USA}
    \thanks{
        Instituto de Computa\c{c}\~{a}o, Universidade Federal do Rio de Janeiro, Rio de Janeiro, Brazil}
    \thanks{
        \texttt{adauto@ic.ufrj.br, adautomn@vcu.edu}}
\and
Viktoriya Bardenova
    \footnotemark[1]
    \thanks{
        \texttt{bardenovav@vcu.edu}}
\and
Yunus Bidav
    \footnotemark[1]
    \thanks{
        \texttt{bidavye@vcu.edu}}
\and
Glenn Hurlbert 
    \footnotemark[1]
    \thanks{
        \texttt{ghurlbert@vcu.edu}}
}


\maketitle

\begin{abstract}
Graph pebbling is a game played on graphs with pebbles on their vertices. 
A pebbling move removes two pebbles from one vertex and places one pebble on an adjacent vertex. 
A configuration $C$ is a supply of pebbles at various vertices of a graph $G$, and a distribution $D$ is a demand of pebbles at various vertices of $G$.
The $D$-pebbling number, $\pi(G, D)$, of a graph $G$ is defined to be the minimum number $m$ such that every configuration of $m$ pebbles can satisfy the demand $D$ via pebbling moves.
The special case in which $t$ pebbles are demanded on vertex $v$ is denoted $D=v^t$, and the \textit{$t$-fold pebbling number}, $\pi_{t}(G)$, equals  $\max_{v\in G}\pi(G,v^t)$.
It was conjectured by Alc\'on, Gutierrez, and Hurlbert that the pebbling numbers of chordal graphs forbidding the pyramid graph can be calculated in polynomial time.
Trees, of course, are the most prominent of such graphs.
In 1989, Chung determined $\p_t(T)$ for all trees $T$.
In this paper, we provide a polynomial-time algorithm to compute the pebbling numbers $\p(T,D)$ for all distributions $D$ on any tree $T$, and characterize maximum-size configurations that do not satisfy $D$.
\end{abstract}
\bigskip



\newpage

\section{Introduction}
\label{s:Introduction}

Graph pebbling is a mathematical game or puzzle that involves moving pebbles along the edges of a connected graph, subject to certain rules. 
The objective of the game is to place a certain number of pebbles on specific vertices of the graph from any sufficiently large configuration of pebbles. 
The original application of graph pebbling involved only the case of placing a single pebble on a certain target. 
As a method to prove results in this realm, the problem was necessarily generalized in \cite{Chung} to place multiple pebbles on a single target.
Similarly, the problem was generalized in \cite{CCFHPST05,HHHtPebb} to placing pebbles on multiple targets.
Good resources on this topic include \cite{AlcoHurlPowers,HurlKent,HurlSedd,MilaClar}, and we list more specific references in Subsections \ref{ss:history} and \ref{ss:motivation}, below.
In this paper, we provide a formula for \rev{the} target pebbling number of a tree, along with an algorithm for calculating it that runs in polynomial time.

We begin with necessary graph theory and pebbling definitions in Subsections \ref{ss:graphdefs} and \ref{ss:pebbdefs}. 
In Subsections \ref{ss:history} and  \ref{ss:motivation} we discuss the relevant history and motivation for our results. 
Subsection \ref{ss:Our Results} contains our two main results regarding the formula and algorithm mentioned above, namely, Theorems \ref{t:Trees} and \ref{t:Algo}.
(It takes us time to develop the concepts and notation necessary to be able to state the formula in Corollary \ref{c:TargetPebbTree}.)
Then we prove these results in Section \ref{s:Proofs} and conclude in Section \ref{s:Conclusion} with additional comments and questions.

\subsection{Graph Definitions}
\label{ss:graphdefs}

In this paper a \textit{graph} $G(V,E)$ is simple and connected.
We write $V$ in place of $V(G)$ unless different graphs need to be distinguished.
The \textit{degree} $\deg_G(v)$ of a vertex $v$ in a graph $G$, is the number of edges incident to $v$.
The {\it indegree} $\deg_G^-(v)$ of a vertex $v$ in a digraph $G$ equals the number of directed edges $uv$ in $G$, and the {\it outdegree} $\deg_G^+(v)$ of a vertex $v$ in a digraph $G$ equals the number of directed edges $vu$ in $G$.
In a connected graph $G$ the distance from $u$ to $v$, denoted $\dist_G(u,v)$ is the length of the shortest $u,v$-path in $G$. 
We ignore the subscripts in the above notations when the graph $G$ is understood.
The \textit{diameter} of a graph $G$ is the maximum distance between any pair of vertices $u$ and $v$, $\diam (G)=\max_{u,v \in V} \dist (u,v)$.

A \textit{tree} is a connected acyclic graph. 
A \textit{leaf} is a vertex of degree 1; every other vertex is {\it interior}. 
Let the set of leaves of a tree $T$ be denoted by $L(T)$.
A graph is {\it chordal} if it has no induced cycle of length 4 or more.
A \textit{simplicial} vertex in a graph is a vertex whose neighbors induce a clique. 
We will denote the set of simplicial vertices by $S(G)$; thus $S(T)=L(T)$ when $T$ is a tree.
Chordal graphs are characterized recursively by either being complete or having a simplicial vertex whose removal leaves a chordal graph.

A \textit{path partition} of a tree $T=(V,E)$ is a set of paths $\cP=\{P_1,\ldots,P_\ell\}$, such that
the paths are edge-disjoint and the union of the paths covers all edges of the tree $V(T)=\cup_{i=1}^ {\ell}V(P_i)$.
For a multiset $a=\{a_1,\ldots,a_m\}$ of nonnegative integers, we will label its indices so that $a_1\ge a_2\ge \cdots\ge a_m$.
For two different multisets $a=\{a_1\ge a_2\ge \cdots\ge a_m\}$ and $b=\{b_1\ge b_2\ge\cdots\ge b_m\}$, we write $a\succ b$ if there is some $j$ such that $a_i=b_i$ for all $i<j$ and $a_j>b_j$. 
When $a\succ b$ we say that $a$ \textit{majorizes} $b$.
A \textit{maximum path partition} of a tree is a path partition whose multiset of path lengths is not majorized by that of any other path partition.
The following algorithm provides maximum path partition $\cP=\{P_1,\ldots,P_\ell\}$ of a rooted tree $(T,r)$ \cite{Chung}.
Set $T_0 = \{r\}$, define $\ell$ to be the number of leaves in $T$ that are different from $r$, and repeat the following two steps for each $1\le i\le \ell$:
(a) define $P_i$ to be a maximum length path in $T$ with one endpoint in $T_{i-1}$, and 
(b) define $T_i = T_{i-1}\cup P_i$.

\subsection{Pebbling Definitions}
\label{ss:pebbdefs}

A {\it pebbling function} $F$ is any function $F:V\rightarrow\zN$.
We define its {\it support} to be $\supp(F)=\{v\in V\mid F(v)>0\}$, with $s(F)=|\supp(F)|$.
For a pebbling function we write $\dF=\{v^{F(v)}\}_{v\in V}$ for its multiset description, and if $\supp(F)=1$ then we ignore the set braces; furthermore, if $F=\{v^t\}$ we write $v^t$ instead of $\{v^t\}$, and if $t=1$ we write $v$ instead of $v^1$.
The {\it size} of $F$ equals $|F|=\sum_{v\in V}F(v)$.
We use the shorthand notation  $\min F = \min_v F(v)$.  
A pebbling function $F$ is called {\it positive} if $\min F>0$, and {\it stacked} if there is a unique $v\in V$ with $F(v)>0$.
For any two functions $F:V\rightarrow\zN$ and $F':V\rightarrow\zN$ the sum $\dF+\dF'=\{v^{F(v)+F'(v)}\}_{v\in V}$. 
In particular, we can view a vertex $v$ as the pebbling function $\{v\}$, so that the sum $F+v$ is obtained from the function $F$ by increasing the value $F(v)$ by 1.
A {\it configuration} $C$ is a pebbling function whose value $C(v)$ represents the number of pebbles at vertex $v$. 
A {\it target} $D$ is a pebbling function whose value $D(v)$ represents the demand of pebbles at vertex $v$.
\rev{The notion of a pebbling function was first introduced in \cite{AlcoHurlPowers}.
One point of doing so is to note that both configurations and targets are pebbling functions; albeit playing different roles (i.e., supply versus demand).
In this way, a set of unified adjectives and notations can be used for both.
Technically, however, it allows one to consider the sum of a configuration and a target, which plays a critical role in the proof of Theorem \ref{t:Path}.}

For adjacent vertices $u$ and $v$, the {\it pebbling step} $u\mapsto v$ consists of removing two pebbles from $u$ and placing one pebble on $v$.
For a configuration $C$ and target $D$, we say that $C$ is $D$-\textit{solvable} (or that there is a $(C, D)$-solution, or that $G$ has a $(C, D)$-solution) if some sequence of pebbling steps places at least $D(v)$ pebbles on each vertex $v$, otherwise $C$ is $D$-\textit{unsolvable}. 
For a solution $\s$ that solves $\{v_1, \ldots, v_k\} \subseteq \dD$, $C[\s]$ denotes the sub-configuration of pebbles used in $\sigma$.
A $(C,D)$-solution $\s=\s_1\ldots\s_m$ is {\it minimal} if, for every $i$, $\s-\s_i$ is not a $(C,D)$-solution. 
For a $D$-unsolvable configuration $C$, we say that $C$ is $D$-\textit{maximal} if adding a single additional pebble to any vertex of $G$ yields a $D$-solvable configuration. 

A configuration $C$ is called $D$-\textit{extremal} if it is $D$-maximal of maximum size.
The $D$-pebbling number, $\pi(G, D)$, of a graph $G$ is defined to be the minimum number $m$ such that $G$ is $(C, D)$-solvable whenever $|C| \geq m$; that is, it is one more than the size of a $D$-extremal configuration. 
The \textit{$t$-fold pebbling number}, $\pi_{t}(G)$, equals $\max_{v\in G}\pi(G,v^t)$; when $t=1$, we simply write $\pi(G)$.
Also, $C$ is  {\it solvable} if it is $r$-solvable for all $r$.

For a $(C,D)$-solution $\s$ and vertex $v\in\dD$, we write $\s_v$ for the set of pebbling steps used to solve $v$; thus $\s$ can be partitioned into $\s_v$ over all $v\in\dD$.
A vertex $v$ is \textit{big} in a configuration $C$ if $C(v) \geq 2$.
$B(C)$ denotes the set of big vertices of a configuration $C$.
A \textit{potential vertex} of a configuration $C$ with respect to a target $D$ is any vertex $v\in B(C)\cup (\supp(C)\cap\dD)$.

A pebbling step $u_i\mapsto v_i$ is $r$-greedy if $\dist(v_i,r)<\dist(u_i,r)$, and a sequence of pebbling steps is $r$-greedy if each of its steps is $r$-greedy.
A $(C,D)$-solution $\s$ is {\it greedy} if, for each $v\in\dD$, $\s_v$ is $v$-greedy.
A configuration $C$ is $D$-\textit{greedy} if it has a greedy $D$-solution.
A graph $G$ is $D$-{\it greedy} if for every configuration $C$ with $|C|\ge\p(G,D)$, there is a greedy $(C,D)$-solution.
For a $(C,D)$-solution $\s = \{u_i\mapsto v_i\}_{i\in I}$, we define its {\it solution subgraph} $G_\s$ to be the digraph with vertices $\cup_i \{u_i,v_i\}_{i\in I}$ and arcs $\cup_i \{u_iv_i\}_{i\in I}$.
The \textit{cost} of $\s$ equals $\cost(\s)=|I|+|D|$; i.e. the number of pebbles used by $\s$.

\subsection{History}
\label{ss:history}

The basic lower and upper bounds for every graph are as follows.

\begin{fct}[\cite{ChanGod,Chung}]
\label{f:BasicBds}
For every graph $G$ we have $\max\{n,$ $2^{\diam(G)}\} \le \pi(G) \leq (n(G)-\diam(G))(2^{\diam(G)}-1)+1$.
\end{fct}

The lower bound $\pi(G) \geq n$ is sharp for cubes \cite{Chung}, 2-connected graphs with a dominating vertex, 3-connected graphs of diameter two \cite{ClaHocHur}, the Petersen graph, and many others, including, probabilistically, almost all graphs \cite{CzHuKiTr} (these are known as {\it Class 0} graphs \cite{ClaHocHur}).
The lower bound $\pi(G) \geq 2^{\diam(G)}$ is sharp for Cartesian products of paths \cite{Chung}, for example.
The upper bound in Fact \ref{f:BasicBds} is sharp for paths and complete graphs.

It has been shown that deciding if $C$ is $r$-solvable is {\sf NP}-complete for general graphs (\cite{HurlKier,MilaClar}), as well as for some specific graph classes (e.g., planar or diameter two \cite{CuLeSiTa} graphs)), while for others is in {\sf P} (e.g., planar and diameter two graphs \cite{LewCusDio}, trees \cite{BCCMW}). 
Deciding if $\p(G)\le K$ takes even longer, in general (in $\p_2^{\sf P}$ --- see \cite{MilaClar}), although formulas for $\p(G)$ that can be calculated in polynomial time are known for many classes of graphs (e.g., Cartesian products of paths \cite{Chung}, trees \cite{BCCMW,Chung}, diameter two graphs \cite{ClaHocHur,PacSneVox}, sufficiently dense graphs \cite{CzygHurl}, split graphs \cite{AlcGutHurSemi}, powers of paths \cite{AlcoHurlPowers}, and others --- see \cite{HurlKent}).
In Theorem \ref{t:Algo} we provide a polynomial-time algorithm for computing $\p(T,D)$.
Some of the motivation for this work lies in the following conjecture.
The {\it pyramid} is the 6-cycle with a triangle added to one of its maximum independent sets.
\rev{We say that a graph is {\it pyramid-free} if it no induced subgraph is isomorphic to a pyramid.}

\begin{cnj}[\cite{AlcoHurlPowers}]
\label{c:PolyChordal}
If $G$ is a pyramid-free chordal graph then $\p(G)$ can be calculated in polynomial time.
\end{cnj}

Suppose $H$ is a spanning subgraph (e.g. a spanning tree) of $G$.
Because any $(C,D)$-solution in $H$ is a $(C,D)$-solution in $G$, we have $\p(G,D)\le \p(H,D)$.
Thus the calculation of pebbling numbers of trees is important.
Chung \cite{Chung} used a maximum $r$-path partition of $T$ to calculate the case in which $D=\{r\}$.
In this paper we solve the general case (see Corollary \ref{c:TargetPebbTree}, below).
In both cases the calculation involves a polynomial algorithm that decomposes the structure of an $n$-vertex tree in relation to the target.
We state in Theorem \ref{t:Algo} that the algorithm runs in $O(\ed{s(D)\cdot n})$ time.

For the case $D=\{r\}$, Chung used a maximum path partition $\cP=\{P_1,\ldots,P_\ell\}$ of a given rooted tree $(T,r)$.
We note that Bunde, et al. \cite{BCCMW} proved that a path partition of any tree can be constructed in linear time.
Let  $l_i$ denote the length of the path $P_i$, $w_i$ be the leaf vertex of $P_i$ that is not in $\cup\_{j<i}P_j$, and define the configuration $\hC_{r,t}$ on $T$ by $\hC_{r,t}(w_1) = t2^{l_1}-1$, $\hC_{r,t}(w_i) = 2^{l_i}-1$ for each $2\le i\le \ell$, and $\hC_{r,t}(x) = 0$ otherwise. 
We will refer to $\hC_{r,t}$ as a {\it Chung} configuration.
\rev{We remark that a Chung configuration is not $t$-fold $r$-solvable --- indeed, no pebble on $w_i$ can reach the endpoint of $P_i$ opposite from $w_i$; i.e., no pebble from $P_i$ can reach $P_{i-1}$.}
We note that $\ed{\hC_{r,t}(w_1)}=\max\hC_{\ed{r,t}}$ and observe that the unique $t$-fold $r$-solution of $\hC_{r,t}+w_1$ uses only the pebbles from $w_1$.
This property will be a key mechanism in our proofs below.
The simplest example of the following theorem is that $\p(P_n)=2^{n-1}$.

\begin{thm} [Chung \cite{Chung}]
\label{t:Chung}
If $(T,r)$ is a rooted tree then $\p_t(T,r)=|\hC_{r,t}|+1$.
\end{thm}

Thus, every rooted tree $(T,r)$ has a $t$-fold $r$-extremal configuration $C$ with $\supp(C)\sse L(T)$.
In fact, more can be said.
The simple, yet powerful lemma below was first observed by Moews.

\begin{lem} [No-Cycle Lemma \cite{Moews}]
\label{l:NCL}
If $\s$ is a minimal $(C,r)$-solution in $G$ then $G_\s$ is acyclic.
\end{lem}

Note that this implies that minimal solutions in trees are greedy.
Furthermore, this implies that if $C$ is $t$-fold $r$-unsolvable on a tree $T$, and $C(v)>0$ for some interior vertex $v$, then so is the configuration $C'$ defined by $C'(v)=C(v)-1$, $C'(u)=C(u)+2$ for some neighbor $u$ of $v$ that is farther from $r$, and $C'(w)=C(w)$ for all other $w$.
In particular, since $|C'|>|C|$, this proves the following fact.

\begin{fct}
\label{f:Leaves}
For every rooted tree $(T,r)$, every $t$-fold $r$-extremal configuration $C$ on $T$ has $\supp(C)\sse L(T)$.
\end{fct}

Later, Alc\'on and Hurlbert generalized this result to chordal graphs.

\begin{thm} [Alc\'on, Hurlbert \cite{AlcoHurlPowers}]
\label{t:Simplicial}
If $G$ is chordal then, for all $r\in V$ and $t\ge 1$, there is a $t$-fold $r$-extremal configuration $C$ with $\supp(C)\sse S(G)$. 
\end{thm}

This led to the following conjectured generalization.

\begin{cnj}
\label{c:Simplicial}
If $G$ is chordal and $C$ is $D$-extremal then there is a $D$-unsolvable configuration $C^*$ with $|C^*|\ge |C|$ and $\supp(C^*)\sse S(G)$.
\end{cnj}

This conjecture had been shown for trees, in Theorem 8 of \cite{CCFHPST05}, in the case that $D(v)>0$ for all $v\in V(G)$.
We prove it for trees and all $D$ in Theorem \ref{t:Trees}.

\subsection{Further Motivation}
\label{ss:motivation}

After Chung generalized the target $r$ to $r^t$, the next general target to be considered was $D=V$.
In $\cite{CCFHPST05}$, the pebbling number $\p(G,V)$ was called the {\it cover} pebbling number.
For example, they proved that $\p(P_n,V)=2^n-1$.
More generally, they defined the functions $\a(v,D)=\sum_{u\in V}D(u)2^{\dist(u,v)}$ for $v\in V$ and $\a(G,D)=\max_{v\in V}\a(v,D)$, and proved that, for positive $D$, $\p(T,D)=\a(T,D)$ for every tree $T$.
The authors also conjectured that the same formula holds for all graphs, which they proved for complete graphs --- the formula simplifies in this case to $\p(K_n,D)=2|D|-\min D$ for positive $D$.
The conjecture was eventually proven, independently, by Vuong and Wyckoff \cite{VW04} and Sjostrand \cite{S05}.

\begin{thm}[\cite{S05,VW04}]
\label{t:CoverPebb}
For positive $D$ on any graph $G$, $\p(G,V) = \a(G,V)$.
\end{thm}

The main technique in proving this fact is showing that any $D$-extremal configuration is stacked.
In general, for non-positive $D$ this is not always true, as we see from Chung configurations on trees.
However, for trees with non-positive targets, we are able to generalize the stacking notion to that of a superstack, which will have one ``dominant'' stack, in the manner of a Chung configuration.

A new tool in attacking the pebbling number of graphs is the Weak Target Conjecture \cite{HHHtPebb}.

\begin{cnj}[Weak Target Conjecture]
   Every graph $G$ satisfies $\pi(G, D) \leq  \pi_{|D|}(G)$ for every target $D$.
\label{c:WTC}
\end{cnj}

Herscovici, et al. \cite{HHHtPebb} showed that the conjecture is true for trees, cycles, complete graphs, and cubes. 
Hurlbert and Seddiq \cite{HurlSedd} verified it for the families of 2-paths and Kneser graphs $K(m,2)$.
Later, the authors of \cite{AlcoHurlPowers} proposed the more powerful Strong Target Conjecture.

\begin{cnj}[Strong Target Conjecture] 
Every graph $G$ satisfies $\pi(G, D) \leq \pi_{|D|}(G) - s(D) + 1$ for every target $D$.
\label{c:STC}
\end{cnj}

They proved that trees and powers of paths\footnote{The $k^{\rm th}$ {\it power} $G^{(k)}$ of a graph $G$ adds to $G$ all edges $xy$ with $\dist_G(x,y)\le k$.} $P_n^{(k)}$
satisfies this stronger conjecture for all $n$ and $k$.
In fact, the authors of \cite{AlcoHurlPowers} used these facts to derive the formula for $\p(P_n^{(k)})$.

The truth of Conjectures \ref{c:WTC} and \ref{c:STC} may prove to be a useful tool in proving results for more general families of graphs. 
For example, it has been conjectured that the pebbling numbers of chordal graphs of a certain type can be calculated in polynomial time (see \cite{AlcoHurlSplit}). 
Furthermore, the use of general targets could be helpful in proving Graham’s conjecture (see \cite{Chung}) that $\pi(G_1\mathbin{\gbox}G_2)\leq\pi(G_1)\pi(G_2)$, where $\gbox$ denotes the Cartesian product of connected graphs $G_1$ and $G_2$. 
Herscovici, et al. \cite{HHHGraham}, generalized this to conjecture that $\pi(G_1\mathbin{\gbox}G_2, D_1 \times D_2)\leq\pi(G_1,D_1)\pi(G_2,D_2)$.

\subsection{Our Results}
\label{ss:Our Results}

In this section we present our primary results, including a verification of Conjecture \ref{c:Simplicial} for trees \rev{and} an algorithm to compute the pebbling numbers $\p(T,D)$ and $\p_t(T)$ that runs in polynomial time.

\begin{thm}
\label{t:Trees}
If $T$ is a tree and $C$ is $D$-extremal then $\supp(C)\sse L(T)$.
\end{thm}

We prove Theorem \ref{t:Trees} in Subsection \ref{ss:Trees}.
From this key result, we develop an additional structure that can be imposed on $D$-extremal configurations --- see Lemma \ref{l:SuperStack}.
It is Lemma \ref{l:SuperStack} that enables us to write a formula for $\p(T,D)$ in Corollary \ref{c:TargetPebbTree}.

\begin{thm}
\label{t:Algo}
For every tree $T$ and target $D$, the values of $\p(T,D)$ and $\p_t(T)$ can be computed in $O(\ed{s(D)\cdot n})$ and $O(n^2)$ time, respectively.
\end{thm}

The proof of Theorem \ref {t:Algo} can be found at the end of Section \ref{s:Proofs}, and also rests on the structure imposed by Lemma \ref{l:SuperStack}.


\section{Proofs}
\label{s:Proofs}

\subsection{Paths} \label{ss:Paths}

Let $F=C + D$ be a pebbling function on a graph. 
Define $\cA = \langle a_1, \ldots, a_m \rangle$ to be the \textit{pebbling arrangement} of $F$, with $m = n +|C|+|D|$, where $\cA = \langle v_1, v_{1,1}, \ldots, v_{1,F(v_1)}, v_2, v_{2,1}, \ldots, v_{2,F(v_2)}, \ldots, v_n, v_{n,1}, \ldots, v_{n,F(v_n)} \rangle$.
For example, in Figure \ref{fig:arrangement_example} we have $G=P_7$ with configuration $\dC=\{v_3^3, v_4^{21}, v_6^5\}$, target $\dD=\{v_1^2, v_2, v_5, v_7^3\}$, and pebbling arrangement 
\[
\cA = 
\langle v_1, v_{1,1}, v_{1,2}, 
v_2, v_{2,1}, 
v_3, v_{3,1}, v_{3,2}, v_{3,3}, 
v_4, v_{4,1}, \ldots, v_{4,21},
v_{5}, v_{5,1}, 
v_6, v_{6,1}, \ldots, v_{6,5}, 
v_7, v_{7,1}, v_{7,2}, v_{7,3}\rangle,
\]
with $m=7+29+7=43$.
We leave it to the reader to check that $C$ is a maximal $D$-unsolvable configuration.

Next, let $\cC$ be the set of maximal $D$-unsolvable configurations $C$ having $\supp(C) \subseteq \{v_1,v_h,v_n\}$. 

\begin{figure}[ht]
\begin{center}
\begin{tikzpicture}[scale=.95,
  typeA/.style={draw, circle, fill=black, minimum size=1.5mm, inner sep=0pt},
  typeB/.style={}
  highlight/.style={line width=4pt, draw=yellow!50, line cap=round},
  highlight1/.style={line width=8pt, draw=blue!50, line cap=round},
  highlight2/.style={line width=4pt, draw=green!50, line cap=round}
  ]
    
  \node[typeA] (A) [label={[label distance=-7mm]90:{\color{black} $v_1$}}] at (0,0) {}; 
  \node[typeA] (B) [label={[label distance=-7mm]90:{\color{black} $v_2$}}] at (2,0) {}; 
  \node[typeA] (C) [label={[label distance=-7mm]90:{\color{black} $v_3$}}] at (4,0) {}; 
  \node[typeA] (D) [label={[label distance=-7mm]90:{\color{black} $v_4=v_h$}}] at (6,0) {}; 
  \node[typeA] (E) [label={[label distance=-7mm]90:{\color{black} $v_5$}}] at (8,0) {}; 
  \node[typeA] (F) [label={[label distance=-7mm]90:{\color{black} $v_6$}}] at (10,0) {}; 
  \node[typeA] (G) [label={[label distance=-7mm]90:{\color{black} $v_7$}}]at (12,0) {}; 
  
  \node[typeB] (H) [label={[label distance=0mm]90:{\color{gren} $C$}}] at (-0.75,0.25) {}; 
  \node[typeB] (I) [label={[label distance=0mm]90:{\color{gren} $C'$}}] at (-0.75,0.75) {};
  \node (J) [label={[label distance=0mm]90:{\color{gren} $C''$}}] at (-0.75,1.25) {}; 
  \node (K) [label={[label distance=0mm]90:{\color{gren} $C^*$}}] at (-0.75,1.75){}; 
  \node (L) [label={[label distance=0mm]90:{\color{gren} $C_1$}}] at (-0.75,2.25){};
  \node (M) [label={[label distance=0mm]90:{\color{gren} $C_n$}}] at (-0.75,2.75) {}; 

  \node[typeB] (N) [label={[label distance=0mm]90:{\color{gren} $0$}}] at (0,0.25) {}; 
  \node[typeB] (O) [label={[label distance=0mm]90:{\color{gren} $0$}}] at (0,0.75) {};
  \node (P) [label={[label distance=0mm]90:{\color{gren} $0$}}] at (0,1.25) {}; 
  \node (Q) [label={[label distance=0mm]90:{\color{gren} $0$}}] at (0,1.75){}; 
  \node (R) [label={[label distance=0mm]90:{\color{gren} $211$}}] at (0,2.25){};
  \node (S) [label={[label distance=0mm]90:{\color{gren} $0$}}] at (0,2.75) {}; 

  \node[typeB] (T) [label={[label distance=0mm]90:{\color{gren} $0$}}] at (2,0.25) {}; 
  \node[typeB] (U) [label={[label distance=0mm]90:{\color{gren} $0$}}] at (2,0.75) {};
  \node (W) [label={[label distance=0mm]90:{\color{gren} $0$}}] at (2,1.25) {}; 
  \node (X) [label={[label distance=0mm]90:{\color{gren} $0$}}] at (2,1.75){}; 
  \node (Y) [label={[label distance=0mm]90:{\color{gren} $0$}}] at (2,2.25){};
  \node (Z) [label={[label distance=0mm]90:{\color{gren} $0$}}] at (2,2.75) {}; 

   \node[typeB] (AA) [label={[label distance=0mm]90:{\color{gren} $3$}}] at (4,0.25) {}; 
  \node[typeB] (AB) [label={[label distance=0mm]90:{\color{gren} $0$}}] at (4,0.75) {};
  \node (AC) [label={[label distance=0mm]90:{\color{gren} $0$}}] at (4,1.25) {}; 
  \node (AD) [label={[label distance=0mm]90:{\color{gren} $0$}}] at (4,1.75){}; 
  \node (AE) [label={[label distance=0mm]90:{\color{gren} $0$}}] at (4,2.25){};
  \node (AF) [label={[label distance=0mm]90:{\color{gren} $0$}}] at (4,2.75) {}; 

   \node[typeB] (AG) [label={[label distance=0mm]90:{\color{gren} $21$}}] at (6,0.25) {}; 
  \node[typeB] (AH) [label={[label distance=0mm]90:{\color{gren} $7$}}] at (6,0.75) {};
  \node (AI) [label={[label distance=0mm]90:{\color{gren} $12$}}] at (6,1.25) {}; 
  \node (AJ) [label={[label distance=0mm]90:{\color{gren} $29$}}] at (6,1.75){}; 
  \node (AK) [label={[label distance=0mm]90:{\color{gren} $0$}}] at (6,2.25){};
  \node (AL) [label={[label distance=0mm]90:{\color{gren} $0$}}] at (6,2.75) {}; 

  \node[typeB] (AM) [label={[label distance=0mm]90:{\color{gren} $0$}}] at (8,0.25) {}; 
  \node[typeB] (AN) [label={[label distance=0mm]90:{\color{gren} $0$}}] at (8,0.75) {};
  \node (AO) [label={[label distance=0mm]90:{\color{gren} $0$}}] at (8,1.25) {}; 
  \node (AP) [label={[label distance=0mm]90:{\color{gren} $0$}}] at (8,1.75){}; 
  \node (AQ) [label={[label distance=0mm]90:{\color{gren} $0$}}] at (8,2.25){};
  \node (AR) [label={[label distance=0mm]90:{\color{gren} $0$}}] at (8,2.75) {}; 

  \node[typeB] (AS) [label={[label distance=0mm]90:{\color{gren} $5$}}] at (10,0.25) {}; 
  \node[typeB] (AT) [label={[label distance=0mm]90:{\color{gren} $5$}}] at (10,0.75) {};
  \node (AU) [label={[label distance=0mm]90:{\color{gren} $0$}}] at (10,1.25) {}; 
  \node (AV) [label={[label distance=0mm]90:{\color{gren} $0$}}] at (10,1.75){}; 
  \node (AW) [label={[label distance=0mm]90:{\color{gren} $0$}}] at (10,2.25){};
  \node (AX) [label={[label distance=0mm]90:{\color{gren} $0$}}] at (10,2.75) {}; 

  \node[typeB] (AY) [label={[label distance=0mm]90:{\color{gren} $0$}}] at (12,0.25) {}; 
  \node[typeB] (AZ) [label={[label distance=0mm]90:{\color{gren} $0$}}] at (12,0.75) {};
  \node (BA) [label={[label distance=0mm]90:{\color{gren} $0$}}] at (12,1.25) {}; 
  \node (BB) [label={[label distance=0mm]90:{\color{gren} $0$}}] at (12,1.75){}; 
  \node (BC) [label={[label distance=0mm]90:{\color{gren} $0$}}] at (12,2.25){};
  \node (BD) [label={[label distance=0mm]90:{\color{gren} $166$}}] at (12,2.75) {}; 

   \node[typeB] (BE) [label={[label distance=0mm]90:{\color{red} $D$}}] at (-0.75,-1.2) {}; 
  \node[typeB] (BF) [label={[label distance=0mm]90:{\color{red} $D'=D''$}}] at (-1,-1.70) {};
  \node (BG) [label={[label distance=0mm]90:{\color{red} $D_L$}}] at (-0.75,-2.2) {}; 
  \node (BH) [label={[label distance=0mm]90:{\color{red} $D^-_R$}}] at (-0.75,-2.7){}; 
  \node (BI) [label={[label distance=0mm]90:{\color{red} $D^+_R$}}] at (-0.75,-3.2){};

   \node[typeB] (BJ) [label={[label distance=0mm]90:{\color{red} $2$}}] at (0,-1.2) {}; 
  \node[typeB] (BK) [label={[label distance=0mm]90:{\color{red} $0$}}] at (0,-1.70) {};
  \node (BL) [label={[label distance=0mm]90:{\color{red} $2$}}] at (0,-2.2) {}; 
  
   \node[typeB] (BM) [label={[label distance=0mm]90:{\color{red} $1$}}] at (2,-1.2) {}; 
  \node[typeB] (BN) [label={[label distance=0mm]90:{\color{red} $0$}}] at (2,-1.70) {};
  \node (BO) [label={[label distance=0mm]90:{\color{red} $1$}}] at (2,-2.2) {}; 

   \node[typeB] (BP) [label={[label distance=0mm]90:{\color{red} $0$}}] at (4,-1.2) {}; 
  \node[typeB] (BQ) [label={[label distance=0mm]90:{\color{red} $0$}}] at (4,-1.70) {};

   \node[typeB] (BR) [label={[label distance=0mm]90:{\color{red} $0$}}] at (6,-1.2) {}; 
  \node[typeB] (BS) [label={[label distance=0mm]90:{\color{red} $0$}}] at (6,-1.70) {};

  \node[typeB] (BT) [label={[label distance=0mm]90:{\color{red} $1$}}] at (8,-1.2) {}; 
  \node[typeB] (BU) [label={[label distance=0mm]90:{\color{red} $1$}}] at (8,-1.70) {};
  \node (BV) [label={[label distance=0mm]90:{\color{red} $1$}}] at (8,-2.7) {}; 

   \node[typeB] (BW) [label={[label distance=0mm]90:{\color{red} $0$}}] at (10,-1.2) {}; 
  \node[typeB] (BX) [label={[label distance=0mm]90:{\color{red} $0$}}] at (10,-1.70) {};

  \node[typeB] (BY) [label={[label distance=0mm]90:{\color{red} $3$}}] at (12,-1.2) {}; 
  \node[typeB] (BZ) [label={[label distance=0mm]90:{\color{red} $3$}}] at (12,-1.70) {};
  \node (CA) [label={[label distance=0mm]90:{\color{red} $3$}}] at (12,-3.2) {}; 
  
  \draw[thick] (A) -- (B);
  \draw[thick] (B) -- (C);
  \draw[thick] (C) -- (D);
  \draw[thick] (D) -- (E);
  \draw[thick] (E) -- (F);
  \draw[thick] (F) -- (G);
\end{tikzpicture}
\end{center}
\caption{A path $P$ with both endpoints in $D$, showing the sequence of configurations $C$, $C'$, ..., and $C_n$, (in {\color{gren} green}), and targets $D$, $D'$, ..., and $D_R^+$ (in {\color{red} red}) that are used in the proof of Theorem \ref{t:Path}.}
\label{fig:arrangement_example}  
\end{figure}

\begin{thm}
\label{t:Path}
If $P$ is a path and $C$ is $D$-unsolvable then there is a $D$-unsolvable configuration $C^*$ with $|C^*|\ge |C|$ and $\supp(C^*)\sse L(P)$.
\end{thm}

\begin{proof}
Let $G=P_n$, with $|D|=t$.
We use induction on $n$ and $t$.

When $n \le 2$ the statement is trivial, and when $t=1$ the statement is proved by Theorem \ref{t:Simplicial} (since $s(D)=1$).
Now we may assume that $n > 2$ and $t > 1$.

Suppose that some endpoint (without loss of generality, $v_1$) is not in $D$. Suppose that the first target is on $v_k$, for some $k$, where $1<k<n$.
If $C_{[1,k-1]}$ solves $v_k$, we use the necessary pebbles to solve $v_k$.
Let $\sigma$ be a minimum $(C_{[1,k-1]},v_k)$-solution.
$C[\s]$ is the sub-configuration of pebbles used in $\sigma$.
Define $C'=C-C[\s]$ and $D'=D-v_k$.
Then $C'$ is not $D'$-solvable, and so by induction on $t$ (since $|D'|<t$) we know that there exists a $D'$-unsolvable configuration $C''$, such that $|C''|=|C'|$ and every interior vertex is empty.
Finally, define $\dC^* = \dC'' + v_1^{|C[\s]|}$.
Then $C^*$ is $D$-unsolvable with every interior vertex empty.
If $C_{[1,k-1]}$ doesn't solve $v_k$, we then remove $C_{[1,k-1]}$.
Since $C_{[1,k-1]}$ doesn't solve $v_k$ we define $C'=C_{[k,n]}$ and set $D'=D$.
We remove vertices $v_1$ through $v_{k-1}$ and obtain a shorter path $P'=P_{[k,n]}$; then $C'$ is $D'$-unsolvable on $P'$.
We can apply induction on the number of vertices of $P'$ to obtain the configuration $C''$, which is $D'$-unsolvable of size $|C'|$ and has $\supp(C'')\subseteq\{v_k,v_n\}$.
Let $\dC^* = v_1^{|C_{[1,k]}| + C''(v_k)} + (\dC''-v_k^{C''(v_k)})$. 
Since $C''$ is $D$-unsolvable, $C^*$ is also $D$-unsolvable.

If both endpoints are in $D$ we define $\cA = \langle a_1, \ldots, a_m \rangle$ be the pebbling arrangement (as shown before the statement of Theorem \ref{t:Path} and in Figure \ref{fig:arrangement_example}) of $C + D$, with $m = n +|C|+|D|$.
Let $\cC$ be the set of maximal $D$-unsolvable configurations $C$ having $\supp(C) \subseteq \{v_1,v_h,v_n\}$. 

Let $j$ be the minimum such that $C_{\langle 1,j \rangle}$ solves $D_{\langle 1,j\rangle}$, and let $\s$ denote this solution.
Define $h=h(j)$ to be such that each pebble $a_{j, i}$ is on vertex $v_h$. 
(In the example of Figure \ref{fig:arrangement_example} we have $h=4$, $|C[\s]|=3+14=17$, and $j=24$, since the $17^{\rm th}$ pebble in $C[\s]$ is $v_{4,14}$, which is the $24^{\rm th}$ element in $\cA$.)
Note that $a_j\in C[\s]$, which implies that $D(h)=0$.
Hence $1<h<n$.
Moreover, the No-Cycle Lemma \ref{l:NCL} implies that every step in $\s$ moves from right to left (i.e. from some $v_k$ to $v_{k-1}$).
Therefore, the configuration $v_h^{|C[\s]|}$ either solves $D_{\langle 1,j\rangle}$ exactly (by being equal to $C[\s]$) or is $D_{\langle 1,j\rangle}$-unsolvable.   
Define $C'=C-C[\s]$, $D'=D-D_{\langle 1,j\rangle} = D_{\langle j+1,m\rangle} = D_{[h,n]}$, and let $\cA'=\langle a'_1,\ldots, a'_{m'}\rangle$ be the pebbling arrangement of $C'+D'$, where $m'=(n-h+1)+|C'|+|D'|$.
(In Figure \ref{fig:arrangement_example} we have $\dC'=\{v_6^7\}$, $\dD'=\{v_5,v_7^3\}$ and \[
\cA' = 
\langle v_1, 
v_2,
v_3, 
v_4, v_{4,1}, v_{4,2}, \ldots, v_{4,7},
v_{5}, v_{5,1}, 
v_6, v_{6,1}, \ldots, v_{6,5}, 
v_7, v_{7,1}, v_{7,2}, v_{7,3}\rangle
\] with $m'=7+12+4=23$.) 
Then $\cA'$ is $D'$-extremal.

We now have that $|D'|<|D|$, so we can use induction on $t$ to know that there is a $D'$-unsolvable configuration $C''$ on $P_{[h,n]}$ with $|C''|=|C'|$ and every interior vertex is empty; i.e. $\supp(C'')\subseteq\{v_h,v_n\}$.
We move $C[\s]$ to $v_{h}$; i.e. define $\dC^*=\dC''+v_{h}^{|C[\s]|}$.
(In Figure \ref{fig:arrangement_example}, we have $|C''|=|C'|=12$, with $\dC'=\{v_4^{12}\}$, $\dD''=\dD'=\{v_6,v_7^3\}$,  $\dC^*=\dC''+v_{h}^{|C[\s]|}$, and $|C^*|=12+17=29$.)
We argue by contradiction that $C^*$ is $D$-unsolvable.
Indeed, suppose that $C^*$ is $D$-solvable.
Then $D_{\langle 1,j\rangle}$ must use all of $v_h^{|C[\s]|}$, and possibly some pebbles from $C''$.
We are then left with a sub-configuration of $C''$ to solve $D''=D_{\langle j+1,m\rangle}$, which is impossible since $C''$ is $D''$-unsolvable.
Hence $C^*$ is $D$-unsolvable.

We now let $C_1$ and $C_n$ be maximal $D$-unsolvable configurations stacked on $v_1$ and $v_n$, respectively, and prove by convexity that either $|C_1|\ge |C^*|$ or $|C_n|\ge |C^*|$. 
(See Figure \ref{fig:arrangement_example}: we need $2 + 1 \cdot 2 + 1 \cdot 2^4 + 3 \cdot 2^6=212$ pebbles stacked on $v_1$ to solve $D$, so $|C_1|=211$; similarly, $|C_7|=166$.
Observe that $29=|C^*|<(|C_1|+|C_7|)/2$.)

Recall that $\cC$ is the set of maximal $D$-unsolvable configurations $C$ having $\supp(C) \subseteq \{v_1,v_h,v_n\}$, which we observe is nonempty because of the existence of $C^*$.
We use the convexity of binary exponentiation to prove that the largest configuration in $\cC$ has no pebbles on $v_h$.
Let $\dC^+ = \dC^*+v_1$; then $C^*$ solves $D$ via the solution we call $\s^+$.
We denote by $D_L$ the multiset of all target vertices $v_i$ with $i<h$ that are solved by $\s^+$ using pebbles from $v_h$; 
$D_R^-$ the multiset of all target vertices $v_j$ with $j>h$ that are solved by $\s^+$ using pebbles from $v_h$; 
and $D_R^+$ the multiset of all target vertices $v_k$ with $k>h$ that are solved by $\s^+$ using pebbles from $v_n$. 
(For example, in Figure \ref{fig:arrangement_example} we have that $\dD_L=\{v_1^2,v_2^1\}$, $\dD_R^-=v_5$, and $\dD_R^+=v_7$.)
Then
\begin{align*}
    |C^*| = |C^+|-1
        &= |C^+(v_h)| + |C^+(v_n)| - 1\\
        &= \left(\sum_{v_i \in D_L} 2^{h-i} + \sum_{v_j \in D_R^-} 2^{j-h} + \sum_{v_k \in D_R^+}2^{n-k}\right) - 1.
\end{align*}
As noted above, we know that $1<h<n$.
Suppose that some $v_j\in D_R^-$ has $j-h<n-j$; i.e. $j<(n+h)/2$.
Then define $C^{*'}=C^*-v_h^{2^{j-h}}+v_n^{2^{n-j}}$.
We see that $C^{*'}\in\cC$ and $|C^{*'}|>|C^*|$, and so we may assume that no such $j$ exists; that is, every $v_j\in \dD_R^-$ has $j\ge (n+h)/2$.
This implies that every $v_k\in \dD_R^+$ has $k\ge (n+h)/2 > (n+1)/2$; i.e. $k-1> n-k$.
Now we have
\begin{align*}
    2|C^*| 
        &= \left(2\sum_{v_i \in D_L}2^{h-i} + 2\sum_{v_j \in D_R^-}2^{j-h} + 2\sum_{v_k \in D_R^+}2^{n-k}\right) - 2\\
        &\le \left(\sum_{v_i \in D_L}2^{n-i} + \sum_{v_j \in D_R^-}2^{j-1} + \sum_{v_k \in D_R^+} (2^{n-k}+2^{n-k})\right) - 2\\
        &< \left(\sum_{v_i \in D_L}(2^{i-1} + 2^{n-i})+\sum_{v_j \in D_R^-}(2^{j-1}+2^{n-j}) + \sum_{v_k \in D_R^+}(2^{k-1}+2^{n-k})\right)-2\\
        &= |C_1|+|C_n|.
\end{align*}
Hence $|C^*|<(|C_1|+|C_n|)/2$, and so either $|C_1|>|C^*|$ or $|C_n|>|C^*|$.
This finishes the proof.
\end{proof}

Let $D$ be a target on a graph $G$ with $\dD=\{v_{i_1},\ldots,v_{i_t}\}$ so that $|D|=t$.
For a vertex $v$ let $\a_G(v,D)$ be the minimum number such that the configuration $v^{\a_G(v,D)}$ solves $D$ on $G$.
(Notice that this is the same function $\a$ as used in Theorem \ref{t:CoverPebb}, although $D$ here is not necessarily positive.)
For example $v_1^{2^k}$ solves $v_{k+1}$ on $P_{k+1}$ but $v_1^{2^k-1}$ does not, and so $\a_{P_{k+1}}(v_1,v_{k+1})=2^k$.
For ease of reading, we shall write $\a$ in place of $\a_G$ when the graph $G$ is understood.

Now let $G=P_n$ be the path $v_1\cdots v_n$.
For $1\le j\le t$ define on $G$ $\dD_j^L=\{v_{i_1},\ldots,v_{i_j}\}$ and $\dD_j^R=\{v_{i_j}\ldots,v_{i_t}\}$.
Additionally, define the configuration
\[\dF_j = \{v_1^{\a_{P_n}(v_1,\dD_j^L)-1},v_n^{\a_{P_n}(v_n,\dD_j^R)-1}\};\]
that is, the number of pebbles on $v_1$ is one shy of the number that would solve all of $\dD_j^L$, and the number of pebbles on $v_n$ is one shy of the number required to solve all of $\dD_j^R$.
Finally, let $f_j=|F_j|$ for $1\le j\le t$, and define $f(D,n)=\max\{f_1,f_t\}+1$.
The following fact is evident.

\begin{fct}
\label{f:StackFormula}
Given the target $\dD=\{v_{i_1},\ldots,v_{i_t}\}$ on the path $P_n$ with vertex $v$, let $d_j=\dist(v,v_{i_j})$. 
Then $\a_G(v,D)=\sum_{j=1}^t 2^{d_j}$.
\end{fct}


\begin{lem}
\label{l:Unimodal} 
Given the target $\dD=\{v_{i_1},\ldots,v_{i_t}\}$ on the path $P_n$, let $f_k$ be defined as above.
Then the sequence $f_1,\ldots, f_t$ is unimodal; in particular, 
$f_h>f_{h-1}$ when $i_h \leq (n+1)/2 $ and $f_h<f_{h-1}$ when $i_h > (n+1)/2 $.
Therefore $\max_k f_k = \max \{f_1,f_t\}$.
\end{lem}

\begin{proof}
From Fact \ref{f:StackFormula} we obtain the following formula:
\[f_k = \sum_{j=1}^{h}\left(2^{\dist(v_1,v_{i_j})}\right)-1 + \sum_{j=h}^{t}\left(2^{\dist(v_n,v_{i_j})}\right)-1.\]
From this it follows that $f_h - f_{h-1} = 2^{\dist(v_1,v_{i_h})} - 2^{\dist(v_n,v_{i_h})}$, which is nonnegative if and only if $\dist(v_1,v_{i_h})\ge \dist(v_n,v_{i_h})$; i.e., $i_h-1\ge n-i_h$.
\end{proof}

\begin{cor}
\label{c:TargetPebbPath}
Suppose that $D$ is a target on the path $P_n$.
Then $\pi(P_n,D)=f(D,n)$.
\end{cor}

\begin{proof}
Let $D$ be a target with $\dD = \{v_{i_1}, \ldots , v_{i_t} \}$, where $i_1 \leq \dots \leq i_t$.
For each $1 \leq j \leq t$ let $D_L(j) = \{v_{i_1}, ..., v_{i_j} \}$ and $D_R(j) = \{v_{i_j}, ..., v_{i_t} \}$.
For each $1 \leq j \leq t$ let $L_j$ be the maximum $D_L(j)$-unsolvable stack on $v_1$, $R_j$ be the maximum $D_R(j)$-unsolvable stack on $v_n$, and $F_j = L_j + R_j$.
By Fact \ref{f:StackFormula} $|L_j| = \a(v_1,D_L(j)) - 1$ and $|R_j| = \a(v_n,D_R(j)) - 1$.
For each $1 \leq j \leq t$ define $f_j = |F_j|$.
By Lemma \ref{l:Unimodal}, among all of the functions $f_j$, the maximum value is achieved at one of the endpoints, so $|C|=\max_{j=1}^t f_j = \max \{f_1, f_t \}$. 
Suppose that $C$ is a $D$-extremal configuration. 
By Theorem \ref{t:Path} we may assume that $supp(C) \subseteq \{v_1,v_n\}$. 
Let $C_L = C_{v_1}$ and $C_R = C_{v_n}$.
Define $j$ to be the maximum such that $C_L$ solves $D_L(j-1)$ but not $D_L(j)$.
Then, by maximality of $C$, $C_R$ solves $D_R(j+1)$ but not $D_R(j)$.
Also, by maximality of $C$, $C_L = L_j$ and $C_R = R_j$; therefore $C = F_j$.
Since $C$ is $D$-extremal, $C \in \{F_1,F_n\}$. 
Hence, $\pi(P_n,D)=|C|+1=\max \{f_1,f_t\}+1=f(D,n)$.
This concludes the proof.
\end{proof}


\subsection{Trees}
\label{ss:Trees}

\begin{proof}[Proof of Theorem \ref{t:Trees}.]
Let $T$ be a tree with $n$ vertices with $|D|=t$.
If $T$ has $n \leq 3$ or $|L(T)|=2$, then $T$ is a path and the result follows from Theorem \ref{t:Path}. So let assume otherwise, that $n\ge 4$ and $|L(T)|\ge 3$.
Additionally, if $s(D)=1$, then the result follows from Fact \ref{f:Leaves}, so we assume that $s(D)\ge 2$, which implies that $t\ge 2$.

Let $C$ be a $D$-extremal configuration --- we may assume that $C$ has the fewest number of interior pebbles among all such configurations.

Let $P$ be any leaf-split vertex path in $T$, where $x$ is a leaf of $P$ and $y$ be the split vertex (opposite of $x$) in $P$.
Define $P'=P-y$ and $T'=T-P'$.
Since $C$ is extremal, $C$ is $(D-w)$-solvable for any $w\in D$. 
Choose $v \in D$ to be the closest to $x$, let $w \in D-v$, and let $\s$ be a minimum $(C,D-w)$-solution.
We now analyze $\s_v$ and define $C'=C-C[\s]$ and $D'=D-v$.

Consider the case that $\s_v$ comes from the $x$-side of $v$.
If $\s_v$ uses a pebble on an interior vertex $z$, then the configuration that replaces that pebble by a pebble on $x$ is a $D$-extremal configuration with fewer interior pebbles, a contradiction.
Hence, $C[\s]$ is on $x$ already.
Since $\s_v$ is minimum, there are no interior pebbles on $P'$.
Since $C'$ is $D'$-extremal in $T$ and $|D'|<|D|$, induction implies that $\supp(C')\sse L(T)$.
Hence $\supp(C)\sse L(T)$.

Now consider the case that $\s_v$ comes from the $y$-side of $v$.
Let $\t$ be the minimum number of pebbles needed to add to $C'$ at $y$ to solve $D$, and define $D''$ on $T'$ by setting $D''(y)=D(y)+\t$ and $D''(u)=D(u)$ for all $u\in T'-y$.
Then $C'$ is $D''$-extremal.
We proceed by induction on the number of leaves since $|L(T')|<|L(T)|$ because $y$ is a split vertex in $T$.
Hence $\supp(C') \sse L(T)$.
Therefore $\supp(C)\sse L(T)$, which completes the proof.
\end{proof}

Let $\cC$ be the set of all $D$-unsolvable configurations $C$ for which $\supp(C)\sse L(T)$.
For $C\in\cC$ and $v\in\supp(C)$, define the {\it stacked} configuration $C_v$ by $C_v(v)=C(v)$ and $C_v(x)=0$ otherwise, and also define $C_v^+$ by $C_v ^+(v)=C_v(v)+1$ and $C_v^+(x)=0$ otherwise.
Next we say that $v$ is a {\it superstack} if the configuration $C_v^+$ is $D$-solvable.
Now let $C^*_v$ be any $D$-extremal configuration in $\cC$ containing $C_v$ --- which we call a {\it superstack configuration} --- and define $\cC^*$ to be set of $C^*_v$ for all leaves $v$.
Note that when $D=\{r^t\}$ we have $C^*_v=\hC_{r,t}$.

As an aside, it is worth noting that superstacks do not characterize extremal configurations.
Indeed, consider $K_{1,3}$ with leaves $r$, $u$, and $v$, and target $D=r^2$.
Then $C_1=\{u^7,v\}$ and $C_2=\{u^5,v^3\}$ are both $D$-extremal, but only $C_1$ is a superstack. 

Define a $(C,D)$-solution $\s$ to be {\it merging} if there are distinct vertices $u$, $v$, and $w$ such that $\s$ contains the pebbling steps $u\mapsto w$ and $v\mapsto w$; otherwise, it is {\it merge-free}.
Note that if $\s$ is a merge-free solution of a single target then $T_\s$ is a path.
Such a vertex $w$ is called a {\it merging} vertex of $\s$; observe that a merging vertex of $\s$ is identified by having $\deg_{T_\s}^-(v)\ge 2$.
For a $(C,D)$-solution $\s$, we define its {\it merge number} $\mu(\s)=\sum_v (\deg_{T_\s}^-(v)-1)$ over all its merging vertices $v$.
For a $D$-solvable configuration $C$, we define its {\it merge number} $\mu(C)=\min_\s\mu(\s)$ over all $(C,D)$-solutions $\s$.

For a $(C,D)$-solution $\s$, we define the {\it source of $\s$} to be $\So(\s)=\{v\mid v\in\supp(C[\s])\}$.
If $\So(\s)=\{w\}$ then we say that $w$ is the source of $\s$.

We record the following evident fact without proof.

\begin{fct}
\label{f:Lex}
Let $a=\{a_1,\ldots,a_m\}$ be a multiset of nonnegative integers, written so that $a_1\ge\cdots\ge a_m$.
Let $1\le j<j'\le m$ and $c\le a_{j'}$, and define $a'=\{a'_1,\ldots,a'_m\}$ by $a'_{j'}=a_{j'}-c$, $a'_j=a_j+c$, and $a'_i=a_i$ otherwise.
Then $a'\succ a$.
\qed
\end{fct}

We will apply this notion to the multiset of values of a configuration $C$: i.e., $\{C(v)\}_{v\in V}$.

\begin{lem}[No-Merging Lemma]
\label{l:NoMerge}
Let $D$ be a target in a tree $T$.
Then there is a $D$-extremal configuration $C$, and a leaf with $C(v)=\max C$, such that any $(C+v,D)$-solution $\s$ is non-merging.
\end{lem}

Observe that paths have these properties by Corollary \ref{c:TargetPebbPath}.

\begin{figure}[ht]
    \centering  
  {
  \begin{tikzpicture}[
  every node/.style={draw, circle, fill=black, minimum size=1.5mm, inner sep=0pt},
  highlight/.style={line width=4pt, draw=yellow!50, line cap=round},
  highlight1/.style={line width=7pt, draw=blue!50, line cap=round},
  highlight2/.style={line width=4pt, draw=gren!50, line cap=round}
  ]
    
  \node (A) [label={[label distance=1mm]90:{\color{red} $5$}}] at (0,0) {}; 
  \node (B) at (-1.41,-1) {}; 
  \node (BB) [label={[label distance=1mm]180:{\color{black} $j$}}] at (1.41,-1) {}; 
  \node (C)  at (-2,-2) {}; 
  \node (CC) [label={[label distance=1mm]270:{\color{gren} $1$}}] at (-.82,-2) {}; 
  \node (CCC) [label={[label distance=1mm]180:{\color{black} $w$}}] at (.82,-2) {}; 
  \node (CCCC) at (2.13,-2) {}; 
  \node (D) [label={[label distance=1mm]270:{\color{gren} $7$}}] at (-2.59,-3) {}; 
  \node (DD) at (0,-3) {}; 
  \node (DDD) [label={[label distance=1mm]270:{\color{gren} $1$}}] at (.82,-3) {}; 
  \node (DDDD) at (1.64,-3) {};
 \node (DDDDD) [label={[label distance=.5mm]270:{\color{gren} $15$}}] [label={[label distance=1mm]180:{\color{black} $z$}}]at (2.82,-3) {}; 
  \node (E) [label={[label distance=.5mm]270:{\color{gren} $11$}}] [label={[label distance=1mm]180:{\color{black} $x$}}] at (-.59,-4) {}; 
  \node (EE) [label={[label distance=1mm]270:{\color{gren} $1$}}] at (1.05,-4) {};
  \node (EEE) [label={[label distance=.5mm]270:{\color{gren} $47$}}] [label={[label distance=1mm]180:{\color{black} $y$}}] at (2.23,-4) {};
  \draw[thick] (A) -- (B) -- (C) -- (D) ;
  \draw[thick] (B) -- (CC) ;
  \draw[thick] (A) -- (BB)-- (CCC) -- (DD) -- (E);
  \draw[thick] (CCC) -- (DD) -- (E);
  \draw[thick] (CCC) -- (DDD);
  \draw[thick] (CCC) -- (DDDD) -- (EEE);
  \draw[thick] (DDDD) -- (EE);
   \draw[thick] (BB)-- (CCCC) -- (DDDDD);
  \end{tikzpicture}
}
$\qquad$
  {\begin{tikzpicture}[
  every node/.style={draw, circle, fill=black, minimum size=1.5mm, inner sep=0pt},
  highlight/.style={line width=4pt, draw=yellow!50, line cap=round},
  highlight1/.style={line width=7pt, draw=blue!50, line cap=round},
  highlight2/.style={line width=4pt, draw=gren!50, line cap=round}
  ]
    
  \node (A) [label={[label distance=1mm]90:{\color{red} $5$}}] at (0,0) {}; 
  \node (B) at (-1.41,-1) {}; 
  \node (BB) [label={[label distance=1mm]180:{\color{black} $j$}}] at (1.41,-1) {}; 
  \node (C)  at (-2,-2) {}; 
  \node (CC) [label={[label distance=1mm]270:{\color{gren} $1$}}] at (-.82,-2) {}; 
  \node (CCC) [label={[label distance=1mm]180:{\color{black} $w$}}] at (.82,-2) {}; 
  \node (CCCC) at (2.13,-2) {}; 
  \node (D) [label={[label distance=1mm]270:{\color{gren} $7$}}] at (-2.59,-3) {}; 
  \node (DD) at (0,-3) {}; 
  \node (DDD) [label={[label distance=1mm]270:{\color{gren} $1$}}] at (.82,-3) {}; 
  \node (DDDD) at (1.64,-3) {};
 \node (DDDDD) [label={[label distance=.5mm]270:{\color{gren} $15$}}] [label={[label distance=-6mm]180:{\color{black} $z$}}]at (2.82,-3) {}; 
  \node (E) [label={[label distance=1mm]270:{\color{gren} $3$}}] [label={[label distance=1mm]180:{\color{black} $x$}}] at (-.59,-4) {}; 
  \node (EE) [label={[label distance=1mm]270:{\color{gren} $1$}}] at (1.05,-4) {};
  \node (EEE) [label={[label distance=.5mm]270:{\color{gren} $55$}}] [label={[label distance=1mm]180:{\color{black} $y$}}] at (2.23,-4) {};
  \draw[thick] (A) -- (B) -- (C) -- (D) ;
  \draw[thick] (B) -- (CC) ;
  \draw[thick] (A) -- (BB)-- (CCC) -- (DD) -- (E);
  \draw[thick] (CCC) -- (DD) -- (E);
  \draw[thick] (CCC) -- (DDD);
  \draw[thick] (CCC) -- (DDDD) -- (EEE);
  \draw[thick] (DDDD) -- (EE);
   \draw[thick] (BB)-- (CCCC) -- (DDDDD);
\end{tikzpicture}

}
    \caption{\rev{A tree $T$ with target $D$ (in {\color{red} red}) and $D$-extremal configurations (in {\color{gren} green}) $C$ (on the left) and $C'$ (on the right), illustrating Lemma \ref{l:NoMerge}.
    Both $w$ and $j$ are merging vertices of the $(C+y,D)$-solution $\s$, with $w$ on the $jy$-path of $T_\s$.
    Then the argument in the proof of Lemma \ref{l:NoMerge} converts $C$ to $C'$ by moving $2(2^2)=8$ pebbles from $x$ to $y$.}}
    \label{fig:MergeConversion}
\end{figure}

\begin{proof}
When $s(D)=1$ Lemma \ref{l:NoMerge} follows from Theorem~\ref{t:Chung}, and so we may assume that $s(D) \geq 2$.

Let $C$ be a $D$-extremal configuration on a tree $T$ with $C\succ C'$ for all $D$-extremal configurations $C'$.
By Theorem \ref{t:Trees} $\supp(C) \subseteq L(T)$.
Choose any $v$ with $C(v)=\max C$.
If $\mu(C+v)=0$ we are done, so we suppose that $\mu(C+v)>0$ and let $\s$ be a $(C+v, D)$-solution, where $\mu(C+v)=\mu(\s)$.
Choose a merging vertex $w$, such that if $u$ is a leaf in $T_{\s}$, then $w$ is the only merge vertex on the $uw$-path of $T_{\s}$.
Then $\So(\s_w)=\{x,y\}$.
Label $x$ and $y$ so that $C(x)\le C(y)$.
(See Figure \ref{fig:MergeConversion}.)

Suppose that $\s$ moves exactly $k$ pebbles from $x$ onto $w$. 
Then $2 \leq k2^{d_x} \leq (C+v)(x) < (k+1)2^{d_x}$.
Define the configuration $C'$ by $C'(x)=(C+v)(x)-k2^{d_x}$, $C'(y)=(C+v)(y)+k2^{d_x}$, and $C'(u)=(C+v)(u)$ for all other $u$.
If $C'$ is not $D$-solvable, that contradicts $C$ being $D$-extremal, and so $C'$ is $D$-solvable, so we assume that $C'$ is $D$-solvable.
Because $C'$ can only contribute an additional $k$ pebbles to $w$, $C'-v$ is $D$-unsolvable.
Also, $|C'-v|=|C|$, and so $C'-v$ is $D$-extremal.
However, $C'-v\succ C$ by Fact \ref{f:Lex}, which is a contradiction.


Therefore, $\mu(C+v)=0$, which completes the proof.
\end{proof}

For vertices $a,b\in T$ denote by $T_{ab}$ the unique $ab$-path in $T$.
Define a {\it bush} to be an orientation of the edges of a tree with every vertex of indegree at most 1 and a unique vertex $u$ of indegree 0.
The vertex of indegree 0 is called the {\it seed}.

Thus we may denote the bush with seed $a$ by $B_a$.

\begin{cor}
\label{c:SingleSource}
Let $D$ be a target in a tree $T$.
Then there is a $D$-extremal configuration $C$, and a leaf $v$ with $C(v)=\max C$, such that the unique $(C+v,D)$-solution $\s$ has $|\So(\s)|=1$.
\end{cor}

Observe that paths have this property by Corollary \ref{c:TargetPebbPath}.

\begin{proof}
The No-Merging Lemma \ref{l:NoMerge} implies that, for $v\in L(T)$, with $C(v)=\max C$ the unique $(C+v,D)$-solution $\s$ is a pairwise disjoint
union of bushes, in which
$\So(\s)$ is the set of its seeds, necessarily containing $v$.
Thus, if $s(D)=1$ then $|\So(\s)|=1$.
Hence we may assume that $s(D)\ge 2$.
We derive the contradiction that if $\s$ contains more than one bush then $C$ is not $D$-extremal.

Let $B_u$ be any bush and let $\cB=\{B_x\}$ be the set of all other bushes.
For bushes $B_x$ and $B_y$ in $\cB$ we write $B_y\prec B_x$ if $u$ and $y$ are in different components of $T-B_x$.
Because $T$ is a tree, the relation $\prec$ is transitive, and since $\cB$ is finite, there exists a minimal bush $B_v\in\cB$.
(See Figure \ref{fig:bush1}.)

\begin{figure}[ht]


\begin{subfigure}{\textwidth}
\begin{center}
\begin{tikzpicture}[
  every node/.style={draw, circle, fill=black, minimum size=1.5mm, inner sep=0pt},
  highlight/.style={line width=4pt, draw=yellow!50, line cap=round},
  highlight1/.style={line width=7pt, draw=blue!50, line cap=round},
  highlight2/.style={line width=4pt, draw=gren!50, line cap=round}
  ]
    
  \node (A) [label={[label distance=-5mm]360:{\color{black} $u$}}] at (0,0) {}; 
  \node (B) at (1,0) {}; 
  \node (BB) at (1,-1) {}; 
  \node (C) [label={[label distance=-6mm]90:{\color{red} $1$}}] at (2,0) {}; 
  \node (D) at (3,0) {}; 
  \node (E) [label={[label distance=-6mm]90:{\color{red} $2$}}] at (4,0) {}; 
  \node (F) at (5,0) {}; 
  \node (G) at (6,0) {}; 
  \node (H) [label={[label distance=-7mm]135:{\color{red} $1$}}] at (7,0) {}; 
  \node (I) at (8,0) {};
  \node (J) [label={[label distance=-5mm]180:{\color{black} $x$}}] at (9,0) {}; 
  \node (K) at (3,-1) {}; 
  \node (KK) [label={[label distance=-6mm]90:{\color{red} $1$}}] at (3,-2) {}; 
  \node (L) at (4.5,-1) {};
  \node (LL) at (4.5,-2) {};
  \node (LLa) at (4,-3) {};
  \node (LLb) [label={[label distance=-6mm]90:{\color{black} $w$}}] at (5,-3) {};
  \node (M) [label={[label distance=-6mm]90:{\color{red} $1$}}] at (5.5,-1) {}; 
  \node (N) at (7,-1) {}; 
  \node (O) [label={[label distance=-6mm]90:{\color{red} $3$}}] at (6.5,-2) {}; 
  \node (P) at (7.5,-2) {};
  \node (Q) [label={[label distance=-6mm]90:{\color{red} $1$}}] at (6.8,-3) {}; 
  \node (R) [label={[label distance=-6mm]90:{\color{red} $2$}}] at (7.5,-3) {};
  \node (S) [label={[label distance=-6mm]90:{\color{black} $v$}}] at (8.2,-3) {};
  
  \draw[line width=1.5pt,blue] (A) -- (B) -- (C) -- (D) -- (E);
  \draw[thick] (E) -- (F) -- (G) -- (H);
  \draw[thick] (B) -- (BB);
  \draw[line width=1.5pt,gren] (H) -- (I) -- (J);
  
  \draw[line width=1.5pt,blue] (D) -- (K) -- (KK);
  
  \draw[line width=1.5pt,brwn] (F) -- (L);
  \draw[line width=1.5pt,brwn] (F) -- (M);
  
  \draw[line width=1.5pt,brwn] (L) -- (LL);
  
  \draw[thick] (LL) -- (LLa);
  \draw[line width=1.5pt,brwn] (LL) -- (LLb);

  \draw[line width=1.5pt,gren] (H) -- (N);
  
  \draw[line width=1.5pt,gren] (N) -- (O);
  \draw[line width=1.5pt,gren] (N) -- (P);

  \draw[line width=1.5pt,gren] (P) -- (Q);
  \draw[line width=1.5pt,orange] (P) -- (R);
  \draw[line width=1.5pt,orange] (P) -- (S);

\end{tikzpicture}

\end{center}
\end{subfigure}
\medskip 


\begin{subfigure}{\textwidth}
\begin{center}
\begin{tikzpicture}[
  every node/.style={draw, circle, fill=black, minimum size=1.5mm, inner sep=0pt},
  highlight/.style={line width=4pt, draw=yellow!50, line cap=round},
  highlight1/.style={line width=7pt, draw=blue!50, line cap=round},
  highlight2/.style={line width=4pt, draw=gren!50, line cap=round}
  ]
    
  \node (A) [label={[label distance=-5mm]360:{\color{black} $u$}}] at (0,0) {}; 
  \node (B) at (1,0) {}; 
  \node (BB) at (1,-1) {}; 
  \node (C) [label={[label distance=-6mm]90:{\color{red} $1$}}] at (2,0) {}; 
  \node (D) at (3,0) {}; 
  \node (E) [label={[label distance=-6mm]90:{\color{red} $2$}}] at (4,0) {}; 
  \node (F) at (5,0) {}; 
  \node (G) at (6,0) {}; 
  \node (H) [label={[label distance=-7mm]135:{\color{red} $1$}}] at (7,0) {}; 
  \node (I) at (7,-1) {};
  \node (J) [label={[label distance=-6mm]90:{\color{black} $x$}}] at (7,-2) {}; 
  \node (K) at (3,-1) {}; 
  \node (KK) [label={[label distance=-6mm]90:{\color{red} $1$}}] at (3,-2) {}; 
  \node (L) at (4.5,-1) {};
  \node (LL) at (4.5,-2) {};
  \node (LLa) at (4,-3) {};
  \node (LLb) [label={[label distance=-6mm]90:{\color{black} $w$}}] at (5,-3) {};
  \node (M) [label={[label distance=-6mm]90:{\color{red} $1$}}] at (5.5,-1) {}; 
  \node (N) at (8,0) {}; 
  \node (O) [label={[label distance=-6mm]90:{\color{red} $3$}}] at (8,-1) {}; 
  \node (P) at (9,0) {};
  \node (Q) [label={[label distance=-6mm]90:{\color{red} $1$}}] at (8.7,-1) {}; 
  \node (R) [label={[label distance=-6mm]90:{\color{red} $2$}}] at (9.3,-1) {};
  \node (S) [label={[label distance=-5mm]180:{\color{black} $v$}}] at (10,0) {};
  
  \draw[line width=1.5pt,blue] (A) -- (B) -- (C) -- (D) -- (E);
  \draw[thick] (E) -- (F) -- (G) -- (H);
  \draw[thick] (B) -- (BB);
  \draw[line width=1.5pt,gren] (H) -- (I) -- (J);
  
  \draw[line width=1.5pt,blue] (D) -- (K) -- (KK);
  
  \draw[line width=1.5pt,brwn] (F) -- (L);
  \draw[line width=1.5pt,brwn] (F) -- (M);
  
  \draw[line width=1.5pt,brwn] (L) -- (LL);
  
  \draw[thick] (LL) -- (LLa);
  \draw[line width=1.5pt,brwn] (LL) -- (LLb);

  \draw[line width=1.5pt,gren] (H) -- (N);
  
  \draw[line width=1.5pt,gren] (N) -- (O);
  \draw[line width=1.5pt,gren] (N) -- (P);

  \draw[line width=1.5pt,gren] (P) -- (Q);
  \draw[line width=1.5pt,orange] (P) -- (R);
  \draw[line width=1.5pt,orange] (P) -- (S);

\end{tikzpicture}

\end{center}
\end{subfigure}
 \caption{A tree $T$ with bushes $B_u$ (in {\color{blue} blue}), $B_x$ (in {\color{gren} green}), $B_w$ (in {\color{brwn} brown}), and $B_v$ (in {\color{orange} orange}), above, and 
 the tree $T$, redrawn with choice of $\prec$-minimal bush $B_v$, below.}
\label{fig:bush1}
\end{figure}

Now define $T'= B_u\cup B_v\cup T_{uv}$.
Note that $T'-E(T_{uv})$ is a disjoint union of bushes $B_z$, each with its seed $z\in V(T_{uv})$.
Now define the pebbling function $D'$ on $T_{uv}$ by $D'(z)=\a_{B_z}(z,D_z)$, where $D_z=D\cap B_z$.
(See Figure \ref{fig:bush2}.)

\begin{figure}[ht]


\begin{subfigure}{\textwidth}
\begin{center}
\begin{tikzpicture}[
  every node/.style={draw, circle, fill=black, minimum size=1.5mm, inner sep=0pt},
  highlight/.style={line width=4pt, draw=yellow!50, line cap=round},
  highlight1/.style={line width=7pt, draw=blue!50, line cap=round},
  highlight2/.style={line width=4pt, draw=gren!50, line cap=round}
  ]
    
  \node (A) [label={[label distance=-5mm]360:{\color{black} $u$}}] at (0,0) {}; 
  \node (B) at (1,0) {}; 
  \node (C) at (2,0) {}; 
  \node (D) [label=above:{\color{blue} $z_u$}] at (3,0) {}; 
  \node (E) at (4,0) {}; 
  \node (F) at (5,0) {}; 
  \node (G) at (6,0) {}; 
  \node (H) at (7,0) {}; 
  \node (K) at (3,-1) {}; 
  \node (KK) [label={[label distance=-6mm]90:{\color{red} $1$}}] at (3,-2) {}; 
  \node (N) at (8,0) {}; 
  \node (P) [label=above:{\color{orange} $z_v$}] at (9,0) {};
  \node (R) [label={[label distance=-6mm]90:{\color{red} $2$}}] at (9.3,-1) {};
  \node (S) [label={[label distance=-5mm]180:{\color{black} $v$}}] at (10,0) {};
  
  \draw[thick] (A) -- (B) -- (C) -- (D) -- (E);
  \draw[thick] (E) -- (F) -- (G) -- (H);
  
  \draw[line width=1.5pt,blue] (D) -- (K) -- (KK);
  



  \draw[thick] (H) -- (N);
  
  \draw[thick] (N) -- (P);

  \draw[line width=1.5pt,orange] (P) -- (R);
  \draw[thick] (P) -- (S);

\end{tikzpicture}

\end{center}
\end{subfigure}
\medskip 


\begin{subfigure}{\textwidth}
\begin{center}
\begin{tikzpicture}[
  every node/.style={draw, circle, fill=black, minimum size=1.5mm, inner sep=0pt},
  highlight/.style={line width=4pt, draw=yellow!50, line cap=round},
  highlight1/.style={line width=7pt, draw=blue!50, line cap=round},
  highlight2/.style={line width=4pt, draw=gren!50, line cap=round}
  ]
    
  \node (A) [label={[label distance=-5mm]360:{\color{black} $u$}}] at (0,0) {}; 
  \node (B) at (1,0) {}; 
  \node (C) [label={[label distance=-6mm]90:{\color{red} $1$}}] at (2,0) {}; 
  \node (D) [label={[label distance=-6mm]90:{\color{red} $4$}}] at (3,0) {}; 
  \node (E) [label={[label distance=-6mm]90:{\color{red} $2$}}] at (4,0) {}; 
  \node (F) at (5,0) {}; 
  \node (G) at (6,0) {}; 
  \node (H) at (7,0) {}; 
  \node (N) at (8,0) {}; 
  \node (P) [label={[label distance=-6mm]90:{\color{red} $4$}}] at (9,0) {};
  \node (S) [label={[label distance=-5mm]180:{\color{black} $v$}}] at (10,0) {};
  
  \draw[thick] (A) -- (B) -- (C) -- (D) -- (E);
  \draw[thick] (E) -- (F) -- (G) -- (H);
  
  



  \draw[thick] (H) -- (N);
  
  \draw[thick] (N) -- (P);

  \draw[thick] (P) -- (S);

\end{tikzpicture}

\end{center}
\end{subfigure}
 \caption{The tree $T'=B_u\cup B_v\cup T_{uv}$ (derived from Figure \ref{fig:bush1}), showing bushes $B_{z_u}$ (in {\color{blue} blue}) and $B_{z_v}$ (in {\color{orange} orange}), with corresponding seeds $z_u$ and $z_v$, above, and the pebbling function $D'$ on $T_{uv}$.}
\label{fig:bush2}
\end{figure}

Clearly, a configuration with support in $\{u,v\}$ solves $D_u\cup D_v$ on $B_u\cup B_v$ if and only if it solves $D'$ on $T_{uv}$.
Let $\s'$ be the $(C+v,D')$-solution on $T_{uv}$ induced by $\s$.
Because $C$ is $D$-extremal on $T$, it is also $D'$-extremal on $T_{uv}$.
Therefore, by Lemma \ref{l:Unimodal} and Corollary \ref{c:TargetPebbPath} we know that a pebble added to a $D'$-extremal configuration on a path consists of a single bush, contradicting that $\s'$ consists of two bushes.
\end{proof}

\begin{lem}
\label{l:SuperStack}
If $D$ is a target on a tree $T$ then there is a $D$-extremal configuration $C\in\cC^*$.
\end{lem}

\begin{proof}
For a target in $D$ we can find a $D$-extremal configuration $C$ and a vertex $v$ with $(C+v,D)$-solution $\s$ with $|\So(\s)|=1$ by Corollary \ref{c:SingleSource}; that is, vertex $v$ is a superstack. 
By Theorem \ref{t:Trees} $\supp(C) \subseteq L(T)$.
Hence, $C\in\cC^*$.
\end{proof}

Let $f(T,D)=\max_{v \in L(T)}{|C_v^*|}+1$.
The following corollary follows immediately from Lemma \ref{c:SingleSource}.

\begin{cor}
\label{c:TargetPebbTree}
Suppose that $D$ is a target on the tree $T$.
Then $\pi(T,D)=f(T,D)$.
\hfill$\Box$
\end{cor}

Now we discuss a method for calculating $f(T,D)$.
Let $T(D)$ be the convex hull of $D$; that is, it is the smallest subtree of $T$ that contains $D$.
Let $V'=L(T(D))=\{v_1,\ldots,v_j\}$.
For each $i$, $1 \leq i \leq j$, let $d'_i=\sum_{v\in D}D(v)2^{\dist(v_i,v)}$.
This equals the minimum number of pebbles placed on $v_i$ that solves $D$.
Let $T_i^-$ be the union of components of $T-v_i$ that are disjoint from $T(D)$, and set $T_i=T_i^-+v_i$.
For each $i$ let $\cP_i$ denote a maximum path partition of the rooted tree $(T_i,v_i)$, $C_i$ denote its corresponding Chung configuration, and $w_i$ be the leaf vertex opposite $v_i$ in the longest path of $\cP_i$.
Then there is some $i$ such that $|C_{w_i}^*|+1=f(T,D)$.
Hence $C_{w_i}^*=C_i-w_i^{C_i(w_i)}+w_i^{C_{w_i}^+(w_i)}$.

\begin{proof}[Proof of Theorem \ref{t:Algo}]
Choose any $v\in\dD$ and run breadth-first search from $v$ on $T$ to find the shortest paths from $v$ to all other $u\in\dD$; the union of those paths equals $T(D)$, whose leaves equal $V'=\{v_1,\ldots,v_j\}$, for some $j\le s(D)$.
This can be done in $O(n)$ steps.
For each $i\le j$, run Dijkstra's algorithm on $T(D)$ to find the distances $\dist(u,v_i)$ for all $u\in V'$; this calculates $d'_i$ for each $i$ and takes $O(\ed{s(D)\cdot n})$ steps.
The construction of all path partitions takes $O(n)$ steps (see \cite{BCCMW}), and then finding the maximum in the definition of $f(T,D)$ takes $j\le s(D)$ steps.
Hence the entire process takes $O(\ed{s(D)\cdot n})$ steps.
\end{proof}


\section{Conclusion}
\label{s:Conclusion}

In light of our results for general targets on trees, we propose expanding Conjecture \ref{c:PolyChordal} to the following.

\begin{cnj}
\label{c:PolyChordalD}
If $G$ is a pyramid-free chordal graph, and $D$ is any target, then $\p(G,D)$ can be calculated in polynomial time.
\end{cnj}

Along these lines, we offer the following proposition that may be of use in pursuit of this conjecture.

\begin{prp}
\label{p:Observation}
If $G$ is chordal and $C$ is maximal $D$-unsolvable then $S(G)-\supp(D)\sse \supp(C)$.
\end{prp}

\begin{proof}
If some $v\in S(G)-\supp(D)$ has $C(v)=0$ then define the configuration $C'$ by $C'(v)=1$ and $C'(u)=C(u)$ for all other $u$.
Because $C$ is $D$-maximal, $C'$ solves $D$; let $\s'$ be a $(C',D)$-solution, which we assume to be acyclic by the No-Cycle Lemma.
Because $C$ is $D$-unsolvable, $\s'$ contains a step $v\mapsto y$, for some $y$, which requires $\s'$ to also contain a step $x\mapsto v$.
Because $\s'$ is acyclic, $x\not=y$.
Now define $\s = \s' - (x\mapsto v\mapsto y) + (x\mapsto y)$ and observe that $\s$ is a $(C,D)$-solution, which is a contradiction.
\end{proof}

Finally, in support of continued pursuit of Conjectures \ref{c:Simplicial}, \ref{c:WTC}, and \ref{c:STC}, we offer some open problems that may aid progress towards this goal.

\begin{prb}
\label{p:NoMerge}
Prove a No-Merging Lemma for specific families of chordal graphs such as split graphs, interval graphs, $k$-trees, etc.
\end{prb}

\begin{prb}
\label{p:Stacking}
Generalize Lemma \ref{l:SuperStack} to specific families of chordal graphs such as split graphs, interval graphs, $k$-trees, etc.
\end{prb}

\bibliographystyle{acm}
\bibliography{refs}

\begin{thebibliography}{10}

\bibitem{AlcoHurlSplit}
{\sc Alc\'on, L., Gutierrez, M., and Hurlbert, G.}
\newblock Pebbling in split graphs.
\newblock {\em SIAM J. Discrete Math. 28}, 3 (2014), 1449--1466.

\bibitem{AlcGutHurSemi}
{\sc Alc\'on, L., Gutierrez, M., and Hurlbert, G.}
\newblock Pebbling in semi-2-trees.
\newblock {\em Discrete Math. 340}, 7 (2017), 1467--1480.

\bibitem{AlcoHurlPowers}
{\sc Alc\'on, L., and Hurlbert, G.}
\newblock Pebbling in powers of paths.
\newblock {\em Discrete Math. 346}, 5-113315 (2023), 20pp.

\bibitem{BCCMW}
{\sc Bunde, D., Chambers, E., Cranston, D., Milans, K., and West, D.}
\newblock Pebbling and optimal pebbling in graphs.
\newblock {\em J. Graph Theory 57\/} (2008), 215--238.

\bibitem{ChanGod}
{\sc Chan, M., and Godbole, A.~P.}
\newblock Improved pebbling bounds.
\newblock {\em Discrete Math. 308}, 11 (2008), 2301--2306.

\bibitem{Chung}
{\sc Chung, F. R.~K.}
\newblock Pebbling in hypercubes.
\newblock {\em SIAM J. Discrete Math. 2}, 4 (1989), 467--472.

\bibitem{CCFHPST05}
{\sc Crull, B., Cundiff, T., Feltman, P., Hurlbert, G., Pudwell, L., Szaniszlo,
  Z., and Tuza, Z.}
\newblock The cover pebbling number of graphs.
\newblock {\em Discrete Math 296\/} (2005), 15--23.

\bibitem{CuLeSiTa}
{\sc Cusack, C., Lewis, T., Simpson, D., and Taggart, S.}
\newblock The complexity of pebbling in diameter two graphs.
\newblock {\em SIAM J. Discrete Math. 26\/} (2012), 919--928.

\bibitem{CzygHurl}
{\sc Czygrinow, A., and Hurlbert, G.}
\newblock Pebbling in dense graphs.
\newblock {\em Australas. J. Combin. 29\/} (2003), 201--208.

\bibitem{CzHuKiTr}
{\sc Czygrinow, A., Hurlbert, G., Kierstead, H., and Trotter, W.~T.}
\newblock A note on graph pebbling.
\newblock {\em Graphs and Combin 18\/} (2002), 219--225.

\bibitem{HHHGraham}
{\sc Herscovici, D., Hester, B., and Hurlbert, G.}
\newblock Generalizations of graham’s pebbling conjecture.
\newblock {\em Discrete Math. 312}, 15 (2012), 2286--2293.

\bibitem{HHHtPebb}
{\sc Herscovici, D., Hester, B., and Hurlbert, G.}
\newblock t-pebbling and extensions.
\newblock {\em Graphs and Combin. 29\/} (2013), 955--975.

\bibitem{ClaHocHur}
{\sc Hurlbert, G., Clarke, T., and Hochberg, R.}
\newblock Pebbling in diameter two graphs and products of paths.
\newblock {\em J. Graph Th. 25}, 2 (1997), 119--128.

\bibitem{HurlKent}
{\sc Hurlbert, G., and Kenter, F.}
\newblock Graph pebbling: A blend of graph theory, number theory, and
  optimization.
\newblock {\em Notices Amer. Math. Soc. 68}, 11 (2021), 1900--1913.

\bibitem{HurlKier}
{\sc Hurlbert, G., and Kierstead, H.}
\newblock Graph pebbling complexity and fractional pebbling.
\newblock {\em Unpublished\/} (2005).

\bibitem{HurlSedd}
{\sc Hurlbert, G., and Seddiq, E.}
\newblock On the target pebbling conjecture.
\newblock {\em Combinatorics, Graph Theory and Computing, Springer Proc. Math.
  Stat. 448\/} (2024), 163--176.

\bibitem{LewCusDio}
{\sc Lewis, T., Cusack, C., and Dion, L.}
\newblock The complexity of pebbling reachability and solvability in planar and
  outerplanar graphs.
\newblock {\em Discrete Applied Mathematics 172\/} (2014), 62--74.

\bibitem{MilaClar}
{\sc Milans, K., and Clark, B.}
\newblock The complexity of graph pebbling.
\newblock {\em SIAM J. Discrete Math. 20}, 3 (2006), 769–798.

\bibitem{Moews}
{\sc Moews, D.}
\newblock Pebbling graphs.
\newblock {\em J. Combin. Theory Ser. B 2\/} (1992), 244--252.

\bibitem{PacSneVox}
{\sc Pachter, L., Snevily, H., and Voxman, B.}
\newblock On pebbling graphs.
\newblock {\em Congr Numer 107\/} (1995), 65--80.

\bibitem{S05}
{\sc Sjostrand, J.}
\newblock The cover pebbling theorem.
\newblock {\em Electron. J. Combin. 12 Note 22\/} (2005), 5.

\bibitem{VW04}
{\sc Vuong, A., and Wyckoff, I.}
\newblock Conditions for weighted cover pebbling of graphs.
\newblock {\em http://arXiv.org/abs/math.CO/0410410\/} (2004).

\end{thebibliography}

\end{document}